\documentclass[preprint,3p,12pt,times]{elsarticle}

\usepackage{amsmath}
\usepackage{amssymb}
\usepackage{mathrsfs}
\def\imnum{\mathrm{i}\,}

\def\diff{\,\mathrm{d}}
\def\pim{\pi_{-}}
\newdefinition{defn}{Definition}
\newdefinition{rem}{Remark}
\newdefinition{exam}{Example}
\newtheorem{thm}{Theorem}
\newtheorem{lem}[thm]{Lemma}

\newproof{proof}{Proof}
\newcommand{\mathbd}[1]{\boldsymbol{#1}}
\renewcommand{\Re}{\mathop{\mathrm{Re}}\nolimits}
\renewcommand{\Im}{\mathop{\mathrm{Im}}\nolimits}
\newcommand{\Rset}{\mathbb{R}}

\DeclareMathOperator{\arcsinh}{arcsinh}
\DeclareMathOperator{\arctanh}{arctanh}
\DeclareMathOperator{\Si}{Si}

\DeclareMathOperator{\Order}{O}
\DeclareMathOperator{\diag}{diag}
\DeclareMathOperator{\Bfunc}{B}
\newcommand{\trans}{\mathrm{T}}
\renewcommand{\pi}{\piup}
\newcommand{\textSE}{\text{\tiny{\rm{SE}}}}
\newcommand{\textDE}{\text{\tiny{\rm{DE}}}}
\newcommand{\SEt}{\psi^{\textSE}}
\newcommand{\DEt}{\psi^{\textDE}}
\newcommand{\SEtInv}{\phi^{\textSE}}
\newcommand{\DEtInv}{\phi^{\textDE}}
\newcommand{\SEtDiv}{\{\SEt\}'}
\newcommand{\DEtDiv}{\{\DEt\}'}
\newcommand{\TransSE}{\mathcal{T}^{\textSE}}
\newcommand{\TransDE}{\mathcal{T}^{\textDE}}
\newcommand{\ProjSE}{\mathcal{P}_N^{\textSE}}
\newcommand{\ProjDE}{\mathcal{P}_N^{\textDE}}
\newcommand{\yn}{\mathbd{y}^{(N)}}
\newcommand{\tyn}{\mathbd{\tilde{y}}^{(N)}}
\newcommand{\tSE}{t^{\textSE}}
\newcommand{\tDE}{t^{\textDE}}
\newcommand{\domD}{\mathscr{D}}
\newcommand{\LC}{\mathbf{L}}
\newcommand{\Hinf}{\mathbf{H}^{\infty}}
\newcommand{\MC}{\mathbf{M}}
\newcommand{\Xsp}{\mathbf{X}}
\newcommand{\Ysp}{\mathbf{Y}}
\newcommand{\Csp}{\mathbf{C}}
\newcommand{\Vol}{\mathcal{V}}
\newcommand{\Jint}{\mathcal{J}}
\newcommand{\JnSE}{\mathcal{J}_N^{\textSE}}
\newcommand{\JnDE}{\mathcal{J}_N^{\textDE}}
\newcommand{\VolSEn}{\Vol_N^{\textSE}}
\newcommand{\VolDEn}{\Vol_N^{\textDE}}

\journal{Elsevier}

\begin{document}

\begin{frontmatter}



\title{Theoretical analysis of
Sinc-collocation methods and
Sinc-Nystr\"{o}m methods
for initial value problems}


\author{Tomoaki Okayama}

\address{Graduate School of Economics, Hitotsubashi University,
2-1 Naka, Kunitachi, Tokyo 186-8601, Japan}
\ead{tokayama@econ.hit-u.ac.jp}

\begin{abstract}
A Sinc-collocation method has been proposed by Stenger,
and he also gave theoretical analysis of the method
in the case of a `scalar' equation.
This paper extends the theoretical results to the case of a `system'
of equations.
Furthermore, this paper proposes more efficient method
by replacing the variable transformation employed in Stenger's method.
The efficiency is confirmed by both of theoretical analysis
and numerical experiments.
In addition to the existing and newly-proposed Sinc-collocation methods,
this paper also gives similar theoretical results
for Sinc-Nystr\"{o}m methods proposed by Nurmuhammad et al.
From a viewpoint of the computational cost,
it turns out that the newly-proposed Sinc-collocation method is
the most efficient among those methods.
\end{abstract}


\end{frontmatter}



\setlength{\abovedisplayskip}{5pt}
\setlength{\belowdisplayskip}{5pt}


\section{Introduction}
\label{sec:introduction}
The concern of this paper is
a system of initial value problems of the form
\begin{equation}
\begin{cases}
\mathbd{y}'(t)= K(t)\mathbd{y}(t)+\mathbd{g}(t),& a\leq t\leq b, \\
\,\mathbd{y}(a)=\mathbd{r},&
\end{cases}
\label{Ini}
\end{equation}
where $K(t)$ is an $n\times n$ matrix,
and $\mathbd{y}(t),\,\mathbd{g}(t),\,\mathbd{r}$
are column vectors of order $n$.
For Eq.~\eqref{Ini},
several numerical methods
based on the Sinc approximation have been developed
so far~\cite{carlson97:_sinc,nurmuhammad05:_numer,stenger93:_numer},
and in general, those methods converge exponentially.
For example, Carlson et al.~\cite{carlson97:_sinc}
proposed a Sinc-collocation method
for Eq.~\eqref{Ini},
and they also claimed that
its convergence rate is $\Order(N^2\exp(-c\sqrt{N}))$.
However,
their method is designed
not for the finite interval $[a,\,b]$
but for the infinite interval $(-\infty,\infty)$ or $[0,\,\infty)$,
and users have to know solution's behavior as $t\to \infty$
to implement the method.
In addition, users also have to know solution's regularity
for implementation.
It is not so practical to assume that
solution's behavior and regularity can be known in prior to computation,
since the solution is an \emph{unknown} function.

Instead of solving Eq.~\eqref{Ini},
Stenger~\cite{stenger93:_numer} firstly transformed the problem
to the Volterra integral equation of the second kind:
\begin{equation}
\mathbd{y}(t)=\mathbd{r}
 + \int_a^t \left\{\mathbd{g}(s) + K(s)\mathbd{y}(s)\right\}\diff s,
\quad a\leq t\leq b,
\label{Vol}
\end{equation}
and derived a Sinc-collocation method for Eq.~\eqref{Vol}.
His method does not require solution's behavior as $t\to\infty$
when the given interval $[a,\,b]$ is finite.
In addition,
he showed theoretically that
solution's regularity needed for implementation
can be found from the \emph{known} functions.
This is an advantage of his method over that of Carlson et al.
Moreover,
he also showed that
the convergence rate of his method is
$\Order(\sqrt{N}\exp(-c\sqrt{N}))$,
where $c$ is the same constant as the result of Carlson et al.
It should be noted that
those theoretical results were shown only in the case
where Eq.~\eqref{Vol} is a scalar equation ($n=1$),
although the method was proposed for a system of equations.
This is because the analysis relies on the
explicit form of the solution $y$
that holds \emph{only} in the scalar case.

The first objective of this study is to extend Stenger's theoretical results
to a system of equations.
That is, this paper shows that even in the case of a system of equations,
solution's regularity actually can be found from the known functions
$K(t)$ and $\mathbd{g}(t)$,
and also shows that
his method converges with the rate:
$\Order(\sqrt{N}\exp(-c\sqrt{N}))$.

The second objective, which is more important in this paper,
is to improve Stenger's method.
The main idea here is replacement of the
variable transformation;
the ``Single-Exponential transformation'' (SE transformation)
is employed in the method of Stenger (and also Carlson et al.),
but it is replaced with the ``Double-Exponential transformation''
(DE transformation)
in the proposed method.
Those two methods are referred to as
the SE-Sinc-collocation method and the
DE-Sinc-collocation method, respectively.
It has been known that
the replacement of the variable transformation
often accelerates the convergence~\cite{mori01:_doubl,sugihara04:_recen},
and in fact,
this paper shows
by theoretical analysis that the rate is drastically improved to
$\operatorname{O}(\exp(-c' N/\log N))$ by the replacement.

The third objective of this study is
to give similar theoretical results
for Sinc-Nystr\"{o}m methods for Eq.~\eqref{Vol}.
The methods
have been proposed by
Nurmuhammad et al.~\cite{nurmuhammad05:_numer},
where both the SE transformation and
the DE transformation are considered.
Those two methods are referred to as
the SE-Sinc-Nystr\"{o}m method and the
DE-Sinc-Nystr\"{o}m method, respectively.
Any convergence analysis has not been given for both
Sinc-Nystr\"{o}m methods,
and as it stands users have no clue to decide
which to choose out of the four methods:
SE/DE-Sinc-collocation methods and SE/DE-Sinc-Nystr\"{o}m methods.
To improve the situation,
this paper analyzes the errors theoretically, and
shows that the convergence rate
of the SE-Sinc-Nystr\"{o}m method is
$\Order(\exp(-c\sqrt{N}))$,
and that of the DE-Sinc-Nystr\"{o}m method is
$\Order(\frac{\log N}{N}\exp(-c' N/\log N))$.

From a viewpoint of the convergence rate,
the DE-Sinc-Nystr\"{o}m method seems to be the best among the four
methods.
From a viewpoint of the computational cost, however,
the DE-Sinc-collocation method
has several advantages (see discussion in Section~\ref{subsec:discuss}).
Moreover, according to the theoretical analysis in this study,
the difference of the convergence rate
between the two is quite small,
and in fact we can confirm it in numerical experiments
(see Section~\ref{sec:numer_exam}).
Therefore, the (proposed) DE-Sinc-collocation method
compares favorably with the DE-Sinc-Nystr\"{o}m method.

The remainder of this paper is organized as follows.
In Section~\ref{sec:review_sinc},
basic definitions and theorems
of Sinc methods are stated.
In Section~\ref{sec:numer_schemes},
four numerical methods to be considered:
SE/DE-Sinc-Nystr\"{o}m methods
and SE/DE-Sinc-collocation methods are described.
Main theoretical results
are stated in Section~\ref{sec:theoret_results},
and their proofs are given in Section~\ref{sec:proofs}.
Numerical examples are presented in Section~\ref{sec:numer_exam}.


\section{Basic definitions and theorems of Sinc methods}
\label{sec:review_sinc}

In this section,
fundamental approximation formulas
derived from the Sinc approximation
are explained with their convergence theorems.

\subsection{Sinc approximation and Sinc indefinite integration over the real axis}
The Sinc approximation is expressed as
\begin{equation}
 F(x)\approx \sum_{j=-N}^N F(jh)S(j,h)(x),\quad x\in\Rset,
 \label{Approx:Sinc-Original}
\end{equation}
where the basis function $S(j,h)(x)$
(the so-called \textit{Sinc function})
is defined by
\[
S(j,h)(x) = \frac{\sin\pi(x/h-j)}{\pi(x/h-j)},
\]
and $h$ is a step size appropriately chosen
depending on a given positive integer $N$.
The Sinc indefinite integration is derived by integrating both sides
of Eq.~\eqref{Approx:Sinc-Original} as
\begin{equation}
\int_{-\infty}^x
 F(t)\diff t
\approx \sum_{j=-N}^N F(jh)\int_{-\infty}^x S(j,h)(t)\diff t
=\sum_{j=-N}^N F(jh) J(j,h)(x),\quad x\in\Rset,
\label{eq:Sinc-indef-int}
\end{equation}
where $J(j,h)(x)$ is defined by
\begin{equation*}
J(j,h)(x)=h\left\{\frac{1}{2}+\frac{1}{\pi}\Si[\pi(x/h-j)]\right\}.
\end{equation*}
Here, $\Si(x)$ is the so-called ``sine integral'' function,
whose routine is available in some numerical libraries
(IMSL, NAG, GSL, and so on).
The approximation~\eqref{eq:Sinc-indef-int}
is called the Sinc indefinite integration.

\subsection{(Generalized) SE-Sinc approximation and SE-Sinc indefinite integration}

When the target interval $[a,\,b]$ is finite,
the Single-Exponential (SE) transformation
\begin{align*}
t &= \SEt(x) = \frac{b-a}{2}\tanh\left(\frac{x}{2}\right)+\frac{b+a}{2},
\\
x &= \SEtInv(t) = \{\SEt\}^{-1}(t) = \log\left(\frac{t-a}{b-t}\right)
\end{align*}
is frequently used with the formulas~\eqref{Approx:Sinc-Original}
and~\eqref{eq:Sinc-indef-int}.
Since this transformation maps $t\in (a,\,b)$ onto $x\in \mathbb{R}$,
we can use~\eqref{Approx:Sinc-Original} and~\eqref{eq:Sinc-indef-int} as
\begin{align}
f(t)=(f\circ\SEt)(\SEtInv(t))
&\approx\sum_{j=-N}^Nf(\tSE_j)S(j,h)(\SEtInv(t)),
\label{eq:SE-Sinc-approx}
\allowdisplaybreaks
\\
\int_a^t f(s)\diff s
=\int_{-\infty}^{\SEtInv(t)}
f(\SEt(x))\SEtDiv(x)\diff x
&\approx\sum_{j=-N}^N f(\tSE_j)\SEtDiv(jh)J(j,h)(\SEtInv(t)),
\label{eq:SE-Sinc-indef}
\end{align}
where $\tSE_j = \SEt(jh)$.
The approximations~\eqref{eq:SE-Sinc-approx}
and~\eqref{eq:SE-Sinc-indef}
are called the SE-Sinc approximation
and the SE-Sinc indefinite integration, respectively.

If $f$ is non-zero at the endpoints $t=a$ and $t=b$,
the SE-Sinc approximation does not work accurately near the endpoints,
because the right hand side of~\eqref{eq:SE-Sinc-approx}
tends to $0$ when $t\to a$ and $t\to b$.
To remedy the issue,
Stenger~\cite{stenger93:_numer} introduced
the auxiliary basis functions
$w_a(t)=(b-t)/(b-a)$ and $w_b(t)=(t-a)/(b-a)$,
and modified the approximation as
\begin{align}
f(t)\approx \ProjSE[f](t):=f(\tSE_{-N})w_a(t) + f(\tSE_{N})w_b(t) +
\sum_{j=-N}^N\TransSE[f](\tSE_j)S(j,h)(\SEtInv(t)),
\label{eq:general-SE-Sinc-approx}
\end{align}
where $\TransSE$ is defined by
$\TransSE[f](t)=f(t) - f(\tSE_{-N})w_a(t) - f(\tSE_{N})w_b(t)$.
Throughout this paper,
the formula~\eqref{eq:general-SE-Sinc-approx}
is called the generalized SE-Sinc approximation.

\subsection{(Generalized) DE-Sinc approximation and DE-Sinc indefinite integration}

Recently,
instead of the SE transformation,
the Double-Exponential (DE) transformation
\begin{align*}
t &= \DEt(x)
   = \frac{b-a}{2}\tanh\left(\frac{\pi}{2}\sinh(x)\right)+\frac{b+a}{2},\\
x &= \DEtInv(t)
   = \{\DEt\}^{-1}(t)
   = \arcsinh\left[\frac{2}{\pi}\arctanh\left(\frac{2 t - b-a}{b-a}\right)\right]
\end{align*}
has been employed by several authors~\cite{mori01:_doubl,muhammad03:_doubl,okayama1x:_improv,sugihara04:_recen,tanaka09:_funct}.
In this case,
the formulas~\eqref{eq:general-SE-Sinc-approx}
and~\eqref{eq:SE-Sinc-indef} are modified as
\begin{align}
&\!\! f(t)\approx \ProjDE[f](t):=f(\tDE_{-N})w_a(t) + f(\tDE_{N})w_b(t) +
\sum_{j=-N}^N\TransDE[f](\tDE_j)S(j,h)(\DEtInv(t)),
\label{eq:general-DE-Sinc-approx}\\
&\!\! \int_a^t f(s)\diff s
\approx\sum_{j=-N}^N f(\tDE_j)\DEtDiv(jh)J(j,h)(\DEtInv(t)),
\label{eq:DE-Sinc-indef}
\end{align}
where $\tDE_j=\DEt(jh)$ and
$\TransDE[f](t)=f(t) - f(\tDE_{-N})w_a(t) - f(\tDE_{N})w_b(t)$.
Throughout this paper,
the formulas~\eqref{eq:general-DE-Sinc-approx}
and~\eqref{eq:DE-Sinc-indef} are called
the generalized DE-Sinc approximation
and the DE-Sinc indefinite integration, respectively.

\subsection{Convergence theorems}

Here let us introduce function spaces needed
to state convergence theorems.
\begin{defn}
\label{Def:Hinf}
Let $\domD$ be a bounded and simply-connected domain
(or Riemann surface).
Then $\Hinf(\domD)$ denotes the family of functions $f$
analytic on $\domD$ such that the norm
$\|f\|_{\Hinf(\domD)}=\sup_{z\in\domD}|f(z)|$
is finite.
\end{defn}
\begin{defn}
\label{Def:LC}
Let $\alpha$ be a positive constant,
and let $\domD$ be a bounded and simply-connected domain
(or Riemann surface)
which satisfies $(a,\,b)\subset \domD$.
Then $\LC_{\alpha}(\domD)$ denotes the family of functions $f\in\Hinf(\domD)$
for which there exists a constant $L$ such that
for all $z$ in $\domD$
\begin{equation}
|f(z)|\leq L |Q(z)|^{\alpha},
 \label{Leq:LC-bounded-by-Q}
\end{equation}
where the function $Q$ is defined by $Q(z)=(z-a)(b-z)$.
\end{defn}
\begin{defn}
\label{Def:MC}
Let $\alpha$ be a constant with $0<\alpha\leq 1$, and let $\domD$
be a domain
with the same conditions as in Definition~\ref{Def:LC}.
Then $\MC_{\alpha}(\domD)$ denotes the family of functions
$f\in\Hinf(\domD)$
for which there exists a constant $M$ such that
for all $z$ in $\domD$
\begin{align*}
|f(z) - f(a)|&\leq M |z - a|^{\alpha},\\
|f(b) - f(z)|&\leq M |b - z|^{\alpha}.
\end{align*}
\end{defn}

In this paper,
$\domD$ is either of the following two domains:
\begin{align*}
\SEt(\domD_d) = \{z=\SEt(\zeta):\zeta\in\domD_d\}
\quad \text{or} \quad
\DEt(\domD_d) = \{z=\DEt(\zeta):\zeta\in\domD_d\},
\end{align*}
where $\domD_d$ is the strip domain defined by
%
$\domD_d=\{\zeta\in\mathbb{C}:|\Im\zeta|<d\}$
for a positive constant $d$
(see also Tanaka et al.~\cite[Figures~1 and~5]{tanaka09:_desinc}
for the concrete shape of the domains).
Convergence theorems for the generalized SE/DE-Sinc approximations
are described as follows.

\begin{thm}[Okayama~{\cite[Theorem~3]{okayamaar:_note}}, see also Stenger~{\cite{stenger93:_numer}}]
\label{Thm:SE-Sinc-Base}
Let $f\in\MC_{\alpha}(\SEt(\domD_d))$ for $d$ with $0<d<\pi$,
let $N$ be a positive integer, and let $h$ be selected by the formula
\begin{equation}
 h=\sqrt{\frac{\pi d}{\alpha N}}. \label{Def:h-SE}
\end{equation}
Then there exists a constant $C$ which is independent of $N$,
such that
\[
\max_{a\leq t\leq b}
\left|f(t)-\ProjSE[f](t)\right|
\leq C \sqrt{N}\exp\left({-\sqrt{\pi d \alpha N}}\right).
\]
\end{thm}
\begin{thm}[Okayama~{\cite[Theorem~6]{okayamaar:_note}}]
\label{Thm:DE-Sinc-Base}
Let $f\in\MC_{\alpha}(\DEt(\domD_d))$ for $d$ with $0<d<\pi/2$,
let $N$ be a positive integer, and let $h$ be selected by the formula
\begin{equation}
h=\frac{\log(2 d N /\alpha)}{N}. \label{Def:h-DE}
\end{equation}
Then there exists a constant $C$ which is independent of $N$,
such that
\[
\max_{a\leq t\leq b}
\left|f(t)-\ProjDE[f](t)\right|
\leq C \exp\left\{\frac{-\pi d N}{\log(2 d N/\alpha)}\right\}.
\]
\end{thm}

Convergence theorems for the SE/DE-Sinc indefinite integration
have also been given as below.
\begin{thm}[Okayama et al.~{\cite[Theorem~2.9]{okayama1x:_error}}]
\label{Thm:SE-Sinc-indef}
Let $(fQ)\in\LC_{\alpha}(\SEt(\domD_d))$ for $d$ with $0<d<\pi$,
let $N$ be a positive integer, and let $h$ be selected
by the formula~\eqref{Def:h-SE}.
Then there exists a constant $C$ which is independent of $N$,
such that
\[
\max_{a\leq t\leq b}
\left|
\int_a^t f(s)\diff s
- \sum_{j=-N}^N f(\tSE_j)\SEtDiv(jh)J(j,h)(\SEtInv(t))
\right|
\leq C \exp\left({-\sqrt{\pi d \alpha N}}\right).
\]
\end{thm}
\begin{thm}[Okayama et al.~{\cite[Theorem~2.16]{okayama1x:_error}}]
\label{Thm:DE-Sinc-indef}
Let $(fQ)\in\LC_{\alpha}(\DEt(\domD_d))$ for $d$ with $0<d<\pi/2$,
let $N$ be a positive integer, and let $h$ be selected
by the formula~\eqref{Def:h-DE}.
Then there exists a constant $C$ which is independent of $N$,
such that
\begin{align*}
\max_{a\leq t\leq b}
\left|
\int_a^t f(s)\diff s
- \sum_{j=-N}^N f(\tDE_j)\DEtDiv(jh)J(j,h)(\DEtInv(t))
\right|
\leq C
\frac{\log(2dn/\alpha)}{N}
 \exp\left\{\frac{-\pi d N}{\log(2 d N/\alpha)}\right\}.
\end{align*}
\end{thm}


\section{Numerical methods}
\label{sec:numer_schemes}

In this section,
four numerical methods to be considered in this paper are described.
First three methods are existing ones:
the SE-Sinc-Nystr\"{o}m method (Section~\ref{subsec:SE-Sinc-Nystroem}),
the DE-Sinc-Nystr\"{o}m method (Section~\ref{subsec:DE-Sinc-Nystroem}),
and
the SE-Sinc-collocation method (Section~\ref{subsec:SE-Sinc-collocation}).
Fourth one is the newly-proposed method:
the DE-Sinc-collocation method (Section~\ref{subsec:DE-Sinc-collocation}).

\subsection{SE-Sinc-Nystr\"{o}m method}
\label{subsec:SE-Sinc-Nystroem}

As for the functions in Eq.~\eqref{Vol},
let $y_i(t)$ and $g_i(t)$ be each element
of the vectors $\mathbd{y}(t)$ and $\mathbd{g}(t)$, respectively,
and let $k_{ij}(t)$ be $(i,j)$-th element of the matrix $K(t)$.
Assume the following conditions:
\begin{quote}\vspace*{-0.4\baselineskip}
\begin{itemize}
 \item[(SE1)] $k_{ij}Q\in\LC_{\alpha}(\SEt(\domD_d))$ $(i=1,\,\ldots,\,n,\,\,\,j=1,\,\ldots,\,n)$,
 \item[(SE2)] $g_i Q\in\LC_{\alpha}(\SEt(\domD_d))$ $(i=1,\,\ldots,\,n)$,
 \item[(SE3)] $y_i\in\Hinf(\SEt(\domD_d))$ $(i=1,\,\ldots,\,n)$,
\end{itemize}\vspace*{-0.4\baselineskip}
\end{quote}
and define $h$ as Eq.~\eqref{Def:h-SE}.
Under those assumptions,
the integral in Eq.~\eqref{Vol} can be approximated by
the SE-Sinc indefinite integration~\eqref{eq:SE-Sinc-indef},
and we have the new approximated equation:
\begin{align}
\yn(t) 
&=\mathbd{r} +
\sum_{j=-N}^N
\left\{\mathbd{g}(\tSE_j)
+K(\tSE_j) \yn(\tSE_j) \right\}\SEtDiv(jh)J(j,h)(\SEtInv(t)).
\label{Def-UsolSE}
\end{align}
In order to determine the approximated solution $\yn$,
we have to obtain the unknown coefficients
on the right hand side in Eq.~\eqref{Def-UsolSE}, i.e.,
\begin{align*}
\mathbd{Y}^{\textSE}
&=[y_1^{(N)}(\tSE_{-N}),\,\ldots,\,y_1^{(N)}(\tSE_N),\,
y_2^{(N)}(\tSE_{-N}),\,\ldots,\,y_2^{(N)}(\tSE_N),\,\ldots,\,
y_n^{(N)}(\tSE_N)
]^{\trans},
\end{align*}
which is a column vector of order $(2N+1)\cdot n$
(notice that $n$ is the number of the system of equations,
and $N$ is the number appearing in $\sum$).
To this end,
let us discretize Eq.~\eqref{Def-UsolSE}
at $(2N+1)$ sampling points: $t = \tSE_i$ ($i=-N,\,\ldots,\,N$),
and derive the system of linear equations.
Let us introduce some notation here.
Let $\sigma_k=1/2+\Si(\pi k)/\pi$,
and let $I^{(-1)}_N$ be a $(2N+1)\times(2N+1)$ matrix defined by
\begin{equation*}
I^{(-1)}_N=[\sigma_{i-j}],\quad i,\,j=-N,\,\ldots,\,N.
\end{equation*}
Let $I_N$ and $I_n$ be identity matrices
of order $(2N+1)$ and $n$, respectively.
Let $D_N^{\textSE}$ and $K_{ij}^{\textSE}$ be
$(2N+1)\times (2N+1)$ diagonal matrices defined by
\begin{align*}
D_N^{\textSE}&=\diag[\SEtDiv(-Nh),\,\ldots,\,\SEtDiv(Nh)],\\
K_{ij}^{\textSE}&=\diag[k_{ij}(\tSE_{-N}),\,\ldots,\,k_{ij}(\tSE_N)],
\end{align*}
and let $[K_{ij}^{\textSE}]$ be an $n\times n$ block of
the matrices $K_{ij}^{\textSE}$.
Furthermore, let $\mathbd{R}$ and $\mathbd{G}^{\textSE}$ be
column vectors of order $(2N+1)\cdot n$ defined by
\begin{align*}
\mathbd{R}&=[r_1,\,\ldots,\,r_1,\,r_2,\,\ldots,\,r_2,\,\ldots,\,r_n]^{\trans},\\
\mathbd{G}^{\textSE}&=[g_1(\tSE_{-N}),\,\ldots,\,g_1(\tSE_N),\,
g_2(\tSE_{-N}),\,\ldots,\,g_2(\tSE_{N}),\,\ldots,\,g_n(\tSE_N)]^{\trans}.
\end{align*}
Then the system of equations to be solved is written as
\begin{align}
 (I_n\otimes I_N - I_n\otimes \{h I^{(-1)}_N D_N^{\textSE}\}[K_{ij}^{\textSE}])
\mathbd{Y}^{\textSE}
 = I_n\otimes \{h I^{(-1)}_N D_N^{\textSE}\}\mathbd{G}^{\textSE}+\mathbd{R},
 \label{linear-eq-SE}
\end{align}
where ``$\otimes$'' denotes the Kronecker product.
By solving the system~\eqref{linear-eq-SE},
the approximated solution $\yn$ is determined
by Eq.~\eqref{Def-UsolSE}.
This procedure is the SE-Sinc-Nystr\"{o}m method.

\subsection{DE-Sinc-Nystr\"{o}m method}
\label{subsec:DE-Sinc-Nystroem}

The important difference from the previous method is
the variable transformation; the SE transformation is replaced with
the DE transformation here.
Assume the following conditions:
\begin{quote}\vspace*{-0.4\baselineskip}
\begin{itemize}
 \item[(DE1)] $k_{ij}Q\in\LC_{\alpha}(\DEt(\domD_d))$ $(i=1,\,\ldots,\,n,\,\,\,j=1,\,\ldots,\,n)$,
 \item[(DE2)] $g_i Q\in\LC_{\alpha}(\DEt(\domD_d))$ $(i=1,\,\ldots,\,n)$,
 \item[(DE3)] $y_i\in\Hinf(\DEt(\domD_d))$ $(i=1,\,\ldots,\,n)$,
\end{itemize}\vspace*{-0.4\baselineskip}
\end{quote}
and define $h$ as Eq.~\eqref{Def:h-DE}.
Under those assumptions,
the integral in Eq.~\eqref{Vol} can be approximated by
the DE-Sinc indefinite integration~\eqref{eq:DE-Sinc-indef},
and we have the new approximated equation:
\begin{align}
\yn(t) 
&=\mathbd{r} +
\sum_{j=-N}^N
\left\{\mathbd{g}(\tDE_j)
+K(\tDE_j) \yn(\tDE_j) \right\}\DEtDiv(jh)J(j,h)(\DEtInv(t)).
\label{Def-UsolDE}
\end{align}
In order to determine the approximated solution $\yn$,
we have to obtain the unknown coefficients:
\begin{align*}
\mathbd{Y}^{\textDE}
&=[y_1^{(N)}(\tDE_{-N}),\,\ldots,\,y_1^{(N)}(\tDE_N),\,
y_2^{(N)}(\tDE_{-N}),\,\ldots,\,y_2^{(N)}(\tDE_N),\,\ldots,\,
y_n^{(N)}(\tDE_N)
]^{\trans},
\end{align*}
which is a column vector of order $(2N+1)\cdot n$.
To this end,
let us discretize Eq.~\eqref{Def-UsolDE}
at $(2N+1)$ sampling points: $t = \tDE_i$ ($i=-N,\,\ldots,\,N$),
and derive the system of linear equations.
Let $D_N^{\textDE}$ and $K_{ij}^{\textDE}$ be
$(2N+1)\times (2N+1)$ diagonal matrices defined by
\begin{align*}
D_N^{\textDE}&=\diag[\DEtDiv(-Nh),\,\ldots,\,\DEtDiv(Nh)],\\
K_{ij}^{\textDE}&=\diag[k_{ij}(\tDE_{-N}),\,\ldots,\,k_{ij}(\tDE_N)],
\end{align*}
and let $[K_{ij}^{\textDE}]$ be an $n\times n$ block of
the matrices $K_{ij}^{\textDE}$.
Furthermore, let $\mathbd{G}^{\textDE}$ be
a column vector of order $(2N+1)\cdot n$ defined by
\begin{align*}
\mathbd{G}^{\textDE}&=[g_1(\tDE_{-N}),\,\ldots,\,g_1(\tDE_N),\,
g_2(\tDE_{-N}),\,\ldots,\,g_2(\tDE_{N}),\,\ldots,\,g_n(\tDE_N)]^{\trans}.
\end{align*}
Then the system of linear equations to be solved is written as
\begin{align}
 (I_n\otimes I_N - I_n\otimes \{h I^{(-1)}_N D_N^{\textDE}\}[K_{ij}^{\textDE}])
\mathbd{Y}^{\textDE}
 = I_n\otimes \{h I^{(-1)}_N D_N^{\textDE}\}\mathbd{G}^{\textDE}+\mathbd{R}.
 \label{linear-eq-DE}
\end{align}
By solving the system~\eqref{linear-eq-DE},
the approximated solution $\yn$ is determined
by Eq.~\eqref{Def-UsolDE}.
This procedure is the DE-Sinc-Nystr\"{o}m method.

\subsection{SE-Sinc-collocation method}
\label{subsec:SE-Sinc-collocation}

Stenger~\cite{stenger93:_numer} developed
the following SE-Sinc-collocation method
independently of Nurmuhammad et al.~\cite{nurmuhammad05:_numer}
(actually more than 10 years before),
but below we find that
it is strongly related to the SE-Sinc-Nystr\"{o}m method
described in Section~\ref{subsec:SE-Sinc-Nystroem}.
Assume the following conditions:
\begin{quote}\vspace*{-0.4\baselineskip}
\begin{itemize}
 \item[(SE1)] $k_{ij}Q\in\LC_{\alpha}(\SEt(\domD_d))$ $(i=1,\,\ldots,\,n,\,\,\,j=1,\,\ldots,\,n)$,
 \item[(SE2)] $g_iQ\in\LC_{\alpha}(\SEt(\domD_d))$ $(i=1,\,\ldots,\,n)$,
 \item[(SE4)] $y_i\in\MC_{\alpha}(\SEt(\domD_d))$ $(i=1,\,\ldots,\,n)$,
\end{itemize}\vspace*{-0.4\baselineskip}
\end{quote}
and define $h$ as Eq.~\eqref{Def:h-SE}.
Let $\mathbd{Y}$ be the solution
of the system of linear equations~\eqref{linear-eq-SE},
and let us write it as
\begin{equation}
\mathbd{Y}=[y_{1,-N},\,y_{1,-N+1},\,\ldots,\,y_{1,N},\,
y_{2,-N},\,y_{2,-N+1},\,\ldots,\,y_{2,N},\,\ldots,\,y_{n,N}
]^{\trans}.
\label{eq:Def-mathbd-Y}
\end{equation}
Then the approximated solution
$\tyn(t)=[\tilde{y}_1(t),\,\ldots,\,\tilde{y}_n(t)]^{\trans}$
is given by
\begin{align}
\tilde{y}^{(N)}_i(t)
=y_{i,-N}w_a(t)+y_{i,N}w_b(t)
+ \sum_{j=-N}^N
\bigl\{y_{i,j} -y_{i,-N}w_a(\tSE_j)
- y_{i,N}w_b(\tSE_j)\bigr\}
S(j,h)(\SEtInv(t)),
\label{Def-SE-Sinc-colloc}
\end{align}
for $i=1,\,\ldots,\,n$.
This procedure is the SE-Sinc-collocation method.

\subsection{DE-Sinc-collocation method (newly proposed)}
\label{subsec:DE-Sinc-collocation}

In view of Sections~\ref{subsec:SE-Sinc-Nystroem}
and~\ref{subsec:DE-Sinc-Nystroem},
it is quite natural to replace
the SE transformation
with the DE transformation in the previous method.
Assume the following conditions:
\begin{quote}\vspace*{-0.4\baselineskip}
\begin{itemize}
 \item[(DE1)] $k_{ij}Q\in\LC_{\alpha}(\DEt(\domD_d))$ $(i=1,\,\ldots,\,n,\,\,\,j=1,\,\ldots,\,n)$,
 \item[(DE2)] $g_iQ\in\LC_{\alpha}(\DEt(\domD_d))$ $(i=1,\,\ldots,\,n)$,
 \item[(DE4)] $y_i\in\MC_{\alpha}(\DEt(\domD_d))$ $(i=1,\,\ldots,\,n)$,
\end{itemize}\vspace*{-0.4\baselineskip}
\end{quote}
and define $h$ as Eq.~\eqref{Def:h-DE}.
Let $\mathbd{Y}$ be the solution
of the system of linear equations~\eqref{linear-eq-DE},
and let us write it as~\eqref{eq:Def-mathbd-Y}.
Then the approximated solution
$\tyn(t)=[\tilde{y}_1(t),\,\ldots,\,\tilde{y}_n(t)]^{\trans}$
is given by
\begin{align}
\tilde{y}^{(N)}_i(t)
=y_{i,-N}w_a(t)+y_{i,N}w_b(t)
+ \sum_{j=-N}^N
\bigl\{y_{i,j} -y_{i,-N}w_a(\tDE_j)
- y_{i,N}w_b(\tDE_j)\bigr\}
S(j,h)(\DEtInv(t)),
\label{Def-DE-Sinc-colloc}
\end{align}
for $i=1,\,\ldots,\,n$.
This procedure is the DE-Sinc-collocation method.

\begin{rem}
\label{rem:sol_regularity}
The assumptions on the solution $\mathbd{y}$,
i.e., (SE3), (DE3), (SE4), (DE4)
seem to be hard to check,
because $\mathbd{y}$ is an \emph{unknown} function to be determined.
In reality, however,
those assumptions are unnecessary,
because both (SE3) and (SE4) can be shown from the conditions
(SE1) and (SE2),
and both (DE3) and (DE4) can be shown from the conditions
(DE1) and (DE2).
To prove the facts is one of the main contributions
of this paper, which is explained next
(Theorems~\ref{thm:SE-sol-regularity} and~\ref{thm:DE-sol-regularity}).
\end{rem}


\section{Theoretical results}
\label{sec:theoret_results}

In this section,
Theoretical results for the four methods in Section~\ref{sec:numer_schemes}
are explained.
The proofs are given in Section~\ref{sec:proofs}.

\subsection{Results on the regularity of the solution}

As described in Remark~\ref{rem:sol_regularity},
the condition on the solution $\mathbd{y}$
is assumed in each scheme.
If the given problem~\eqref{Ini} is a `scalar' equation
($n=1$),
the following result has been known.

\begin{thm}[Stenger et al.~{\cite[Theorem 2.3]{stenger99:_ode_ivp}}]
Let $n=1$, and let
the assumptions {\rm (SE1)} and {\rm (SE2)} be fulfilled.
Then the initial-value problem~\eqref{Ini}
has a unique solution $y_1\in\MC_{\alpha}(\SEt(\domD_d))$.
\end{thm}

This theorem shows the condition (SE4),
and since $\MC_{\alpha}(\SEt(\domD_d))\subset \Hinf(\SEt(\domD_d))$,
the condition (SE3) is also shown.
In this paper, the same result is shown
in the case of a system of equations (for both SE and DE).

\begin{thm}
\label{thm:SE-sol-regularity}
Let the assumptions {\rm (SE1)} and {\rm (SE2)} be fulfilled.
Then the initial-value problem~\eqref{Ini}
has a unique solution $\mathbd{y}$
with
$y_i\in\MC_{\alpha}(\SEt(\domD_d))$ for $i=1,\,\ldots,\,n$.
\end{thm}
\begin{thm}
\label{thm:DE-sol-regularity}
Let the assumptions {\rm (DE1)} and {\rm (DE2)} be fulfilled.
Then the initial-value problem~\eqref{Ini}
has a unique solution $\mathbd{y}$
with
$y_i\in\MC_{\alpha}(\DEt(\domD_d))$ for $i=1,\,\ldots,\,n$.
\end{thm}

\subsection{Results on convergence of the numerical solutions}

In the case of a `scalar' equation,
the convergence of the SE-Sinc-collocation method
is analyzed as follows.
In what follows, $C$ denotes a constant independent of $N$.

\begin{thm}[Stenger~{\cite[pp.~446--447]{stenger93:_numer}}]
Let $n=1$, and let
the assumptions {\rm (SE1)} and {\rm (SE2)} be fulfilled.
Then, for all $N$ sufficiently large,
the system~\eqref{linear-eq-SE} is uniquely solvable,
and the error of the numerical solution $\tilde{y}_1$
of Eq.~\eqref{Def-SE-Sinc-colloc}
is estimated as
\begin{equation*}
\max_{a\leq t\leq b}|y_1(t)-\tilde{y}_1(t)|
\leq C\sqrt{N}\exp\left({-\sqrt{\pi d \alpha N}}\right).
\end{equation*}
\end{thm}

This paper extends the result to a system of equations,
and to the DE-Sinc-collocation method.

\begin{thm}
\label{thm:SE-Sinc-converge}
Let
the assumptions {\rm (SE1)} and {\rm (SE2)} be fulfilled.
Then, for all $N$ sufficiently large,
the system~\eqref{linear-eq-SE} is uniquely solvable,
and the error of the numerical solution $\tyn$
of Eq.~\eqref{Def-SE-Sinc-colloc}
is estimated as
\begin{equation*}
\max_{i=1,\,\ldots,\,n}\left\{
\max_{a\leq t\leq b}|y_i(t)-\tilde{y}_i^{(N)}(t)|\right\}
\leq C\sqrt{N}\exp\left({-\sqrt{\pi d \alpha N}}\right).
\end{equation*}
\end{thm}
\begin{thm}
\label{thm:DE-Sinc-converge}
Let
the assumptions {\rm (DE1)} and {\rm (DE2)} be fulfilled.
Then, for all $N$ sufficiently large,
the system~\eqref{linear-eq-DE} is uniquely solvable,
and the error of the numerical solution $\tyn$
of Eq.~\eqref{Def-DE-Sinc-colloc}
is estimated as
\begin{equation*}
\max_{i=1,\,\ldots,\,n}\left\{
\max_{a\leq t\leq b}|y_i(t)-\tilde{y}_i^{(N)}(t)|\right\}
\leq C \exp\left\{\frac{-\pi d N}{\log(2 d N/\alpha)}\right\}.
\end{equation*}
\end{thm}

Furthermore, this paper also shows the convergence
of the SE/DE-Sinc-Nystr\"{o}m methods.

\begin{thm}
\label{thm:converge-SE-Sinc-Nystroem}
Let
the assumptions {\rm (SE1)} and {\rm (SE2)} be fulfilled.
Then, for all $N$ sufficiently large,
the system~\eqref{linear-eq-SE} is uniquely solvable,
and the error of the numerical solution $\yn$
of Eq.~\eqref{Def-UsolSE}
is estimated as
\begin{equation*}
\max_{i=1,\,\ldots,\,n}\left\{
\max_{a\leq t\leq b}|y_i(t)-y_i^{(N)}(t)|\right\}
\leq C\exp\left({-\sqrt{\pi d \alpha N}}\right).
\end{equation*}
\end{thm}
\begin{thm}
\label{thm:converge-DE-Sinc-Nystroem}
Let
the assumptions {\rm (DE1)} and {\rm (DE2)} be fulfilled.
Then, for all $N$ sufficiently large,
the system~\eqref{linear-eq-DE} is uniquely solvable,
and the error of the numerical solution $\yn$
of Eq.~\eqref{Def-UsolDE}
is estimated as
\begin{equation*}
\max_{i=1,\,\ldots,\,n}\left\{
\max_{a\leq t\leq b}|y_i(t)-y_i^{(N)}(t)|\right\}
\leq C
\frac{\log(2 d n/\alpha)}{N}
 \exp\left\{\frac{-\pi d N}{\log(2 d N/\alpha)}\right\}.
\end{equation*}
\end{thm}

\subsection{Discussion about the performance}
\label{subsec:discuss}

In view of the convergence rates shown above,
the DE-Sinc-Nystr\"{o}m method seems to be the best,
and this was then followed by
the DE-Sinc-collocation method,
the SE-Sinc-Nystr\"{o}m method,
and the SE-Sinc-collocation method.
However,
the DE-Sinc-collocation method (the second one) can be considered as
the best, or at least as useful as the DE-Sinc-Nystr\"{o}m method,
for the following reasons.
Firstly,
the difference of convergence between
the DE-Sinc-Nystr\"{o}m method
and the DE-Sinc-collocation method is quite small,
and actually it is almost indistinguishable in the numerical experiments
(see Figures~\ref{Fig:example1}--\ref{Fig:example2}
in Section~\ref{sec:numer_exam}).
Secondly,
compared to the the approximate solution of the
DE-Sinc-collocation method $\tyn$ (Eq.~\eqref{Def-DE-Sinc-colloc}),
that of the DE-Sinc-Nystr\"{o}m method $\yn$ (Eq.~\eqref{Def-UsolDE})
has time-consuming terms to evaluate.
All of the basis functions in $\tyn$ are elementary functions,
whereas the basis functions $J(j,h)$ in $\yn$
includes the special function $\Si(x)$.
Furthermore, $\tyn$ can be computed with $\Order(nN)$,
but $\yn$ needs $\Order(n^2 N)$
because a matrix-vector product is included in $\yn$.
Therefore, from the viewpoint of the computational cost,
the DE-Sinc-collocation method is better
than the DE-Sinc-Nystr\"{o}m method
(see also Table~\ref{tab:comp}).


\section{Numerical results}
\label{sec:numer_exam}

In this section, numerical examples of the SE/DE-Sinc-Nystr\"{o}m
methods and the SE/DE-Sinc-collocation methods are presented.
The computation was done on Mac OS X 10.6,
Mac Pro two 2.93~GHz 6-Core Intel Xeon with 32 GB DDR3 ECC SDRAM.
The computation programs were implemented
in C++ with double-precision floating-point arithmetic,
and compiled by GCC 4.0.1 with no optimization.
The linear systems~\eqref{linear-eq-SE} and~\eqref{linear-eq-DE}
are solved by using the LU decomposition.
In what follows, $\pim$ denotes an arbitrary positive number
less than $\pi$, and it was set as $\pim=3.14$ in actual computation.
Firstly, let us consider the following two examples.

\begin{exam}
\label{exam:3}
Consider the following initial value problem
(the Halm equation~\cite{polyanin03:_handb})
over the interval $[0,\,1]$:
\[
 (1+x^2)^2 y''(t) -2 y = 0,\quad
 y(0)=0,\quad y'(0)=1,
\]
which is equivalent to the system
\begin{align*}
&&y_1'(t) &= y_2(t),& y_1(0)&=0,&&\\
&&y_2'(t) &= \frac{2}{(1+x^2)^2} y_1(t), & y_2(0)&=1,&&
\end{align*}
whose solution is $y_1(t)=\sqrt{1+x^2}\sinh(\arctan x)$, $y_2(t)=y_1'(t)$.
\end{exam}
\begin{exam}\label{exam:2}
Consider the following initial value problem over the interval $[0,\,2]$:
\begin{align*}
&&y_1'(t) &= -y_1(t) + \frac{1}{2\sqrt{t}} y_2(t),& y_1(0)&=0,&&\\
&&y_2'(t) &= -\frac{1}{\sqrt{t}}y_1(t), & y_2(0)&=1,&&
\end{align*}
whose solution is $y_1(t)=\sqrt{t}\exp(-t)$, $y_2(t)=\exp(-t)$.
\end{exam}

\begin{figure}
\begin{center}
 \begin{minipage}{0.45\linewidth}
  \includegraphics[width=\linewidth]{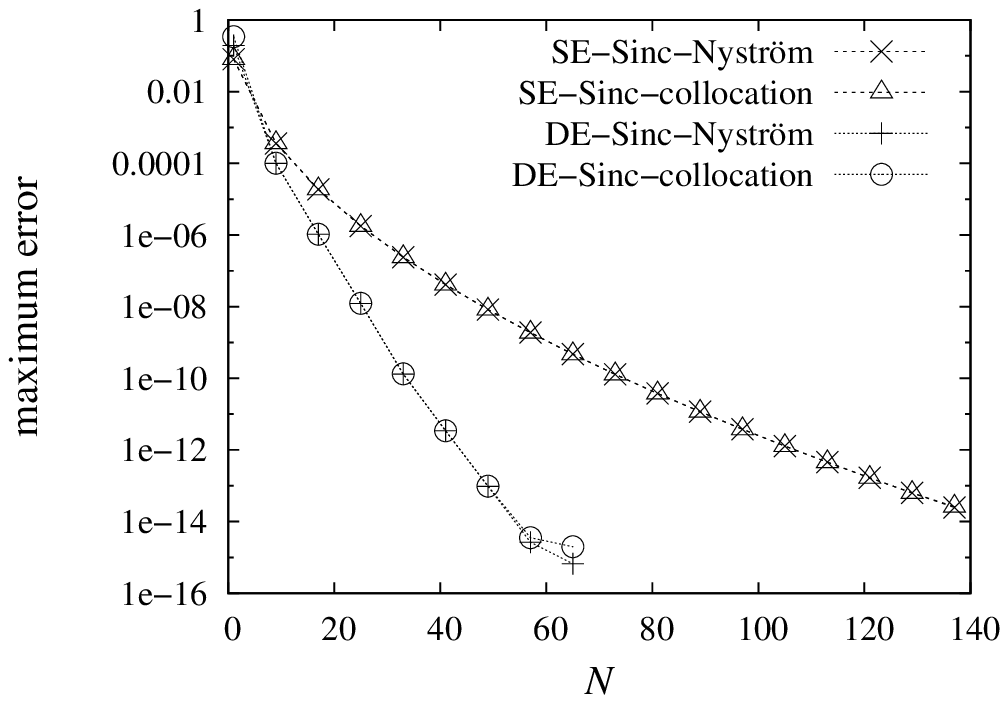}
  \caption{Errors in Example~\ref{exam:3}.}
  \label{Fig:example1}
 \end{minipage}
 \begin{minipage}{0.45\linewidth}
  \includegraphics[width=\linewidth]{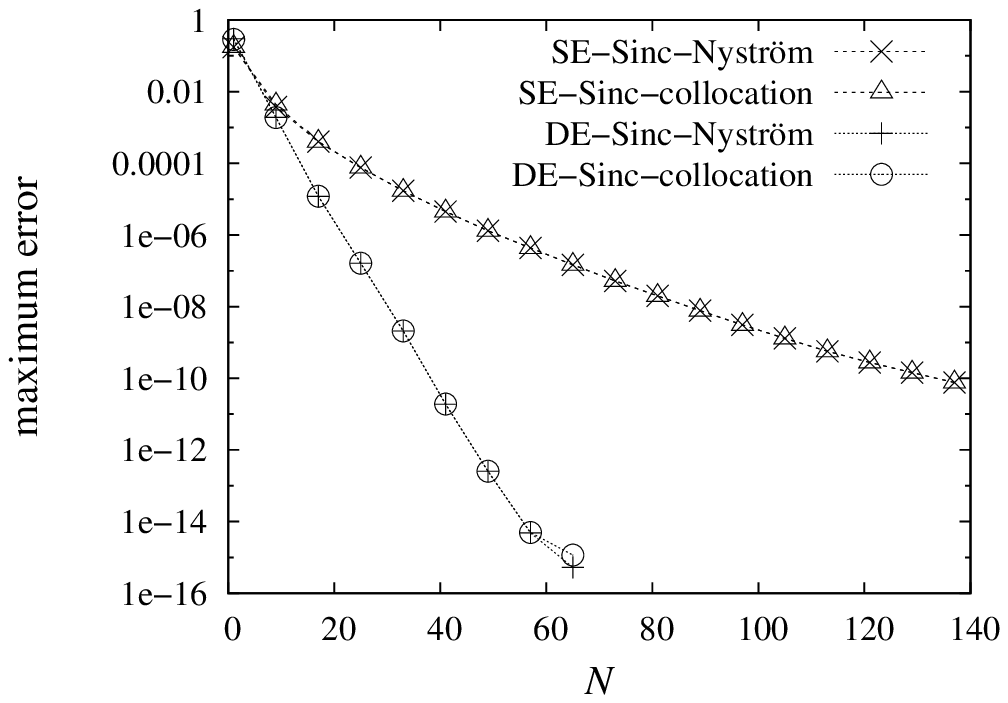}
  \caption{Errors in Example~\ref{exam:2}.}
  \label{Fig:example2}
\end{minipage}
\end{center}
\end{figure}
As for Example~\ref{exam:3}, the conditions (SE1) and (SE2) are satisfied
with $\alpha=1$ and $d=3\pim/4$.
In the DE case, let us set $p=\pim/(2\log 2)$ and
\begin{align*}
q &=\sqrt{\left\{1 + 7p^2 + \sqrt{(1 + 7p^2)^2 + (6p)^2}\right\}/2},
\end{align*}
and furthermore
set $x =-\left(1 - q\right)/(4p)$,
$y =3\left(1 -(1/q)\right)/4$, and
$d_{-} =\arcsin(y/\sqrt{x^2+y^2})$.
Then, the conditions (DE1) and (DE2) are satisfied with $\alpha=1$
and $d=d_{-}$.
As for Example~\ref{exam:2}, which is a harder example because of
the singularity at the origin,
(SE1) and (SE2) are satisfied
with $\alpha=1/2$ and $d=\pim$,
and (DE1) and (DE2) are satisfied with $\alpha=1/2$ and $d=\pim/2$.
The numerical errors are plotted
in Figures~\ref{Fig:example1} and~\ref{Fig:example2}, respectively.
In the graphs,
``maximum error'' denotes
the maximum absolute error at 999 equally-spaced points (say $t_l$)
on the interval $[a,\,b]$, i.e.,
\[
 \text{maximum error}
= \max_{i=1,\,\ldots,\,n}\Bigl\{
\max_{l=1,\,\ldots,\,999}|y_i(t_l) - \widehat{y}_i(t_l)|\Bigr\},
\]
where $\widehat{y}_i$ means each numerical solution.
From both figures, we can confirm the results
of Theorems~\ref{thm:SE-Sinc-converge}--\ref{thm:converge-DE-Sinc-Nystroem}.
More precisely as for the (newly-proposed)
DE-Sinc-collocation method,
its convergence rate is actually
much higher than that of the SE-Sinc-collocation method.
As described in Section~\ref{subsec:discuss},
although the theoretical rate
of the DE-Sinc-Nystr\"{o}m method
is a bit higher than that of the DE-Sinc-collocation method,
both rates are almost
indistinguishable in the numerical results.
Moreover,
as seen in Table~\ref{tab:comp},
the DE-Sinc-Nystr\"{o}m methods needs times
twice as much as the DE-Sinc-collocation method
to obtain $10^{-8}$ accuracy
(the same applies in the SE case).
At least from the result, we can conclude that the DE-Sinc-collocation
method is the most efficient.

\begin{table}\vskip-0.4\baselineskip
\caption{Computation times and $N$ needed to obtain $10^{-8}$ accuracy in Example~\ref{exam:2}.}
\label{tab:comp}
\begin{small}
\begin{center}
\begin{tabular}{ccccc}\hline
& SE-Sinc-Nystr\"{o}m & SE-Sinc-collocation &
DE-Sinc-Nystr\"{o}m & DE-Sinc-collocation \\ \hline
$N$ & 87 & 87 & 31 & 31 \\
time [s] & 0.281 & 0.137 & 0.107 & 0.050 \\ \hline
\end{tabular}
\end{center}\vskip-\baselineskip
\end{small}
\end{table}

In the examples above, all the assumptions (SE1), (SE2),
(DE1), and (DE2) are satisfied with some $\alpha$ and $d$.
Let us have a look at another case here.

\begin{exam}
\label{exam:4}
Set a function $F$ as
$F(t)=\sqrt{\cos(4\arctanh t) + \cosh(\pi)}$, and
consider the following initial value problem
over the interval $[-1,\,1]$:
\begin{align*}
&&y_1'(t) &= -\frac{2[t F^2(t) +\sin(4\arctanh t)]}{F(t)}
y_2(t) ,& y_1(-1)&=0,&&\\
&&y_2'(t) &= \frac{2[t F^2(t) +\sin(4\arctanh t)]}{F(t)}y_1(t), & y_2(-1)&=1,&&
\end{align*}
whose solution is $y_1(t)=\sin[(1-t^2)F(t)]$,
$y_2(t)=\cos[(1-t^2)F(t)]$.
\end{exam}

This is a quite hard example to solve numerically,
due to the bad behavior of $F$ at $t=\pm 1$
(non-regular points are densely distributed around the endpoints).
Fortunately, the assumptions (SE1) and~(SE2) are satisfied
with $\alpha=1$ and $d=\pim/2$,
but (DE1) and~(DE2) are not satisfied with any $d>0$
(we easily see $\alpha=1$, though).
Therefore,
Theorems~\ref{thm:DE-Sinc-converge}
and~\ref{thm:converge-DE-Sinc-Nystroem} cannot be used in this case.
However, according to the recent result~\cite{okayamant:_de_sinc},
even in such a case, DE's methods may achieve the same convergence rate
with that of SE, by setting $d=\arcsin(d_{\textSE}/\pi)$,
where $d_{\textSE}$ denotes SE's $d$.
We can in fact observe it in Figure~\ref{Fig:example4};
DE's methods seem to converge with the similar rate to that of SE.
Since the computational cost is the same as that of the previous examples,
we can consider that the DE-Sinc collocation method still
keeps the lead even in this case.


\begin{figure}
\begin{center}
 \begin{minipage}{0.45\linewidth}
  \includegraphics[width=\linewidth]{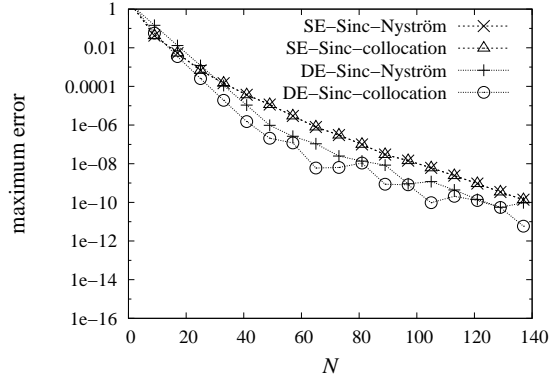}
  \caption{Errors in Example~\ref{exam:4}.}
  \label{Fig:example4}
 \end{minipage}
\end{center}
\end{figure}


\section{Proofs}
\label{sec:proofs}

\subsection{Proofs on the regularity of the solution}

The idea here is to apply the standard contraction mapping theorem,
which holds not only in the scalar case but also in the case of
a system of equations.
Set $\Xsp=\{\Hinf(\domD)\}^n$
and $\Ysp=\{\MC_{\alpha}(\domD)\}^n$,
and define
$\|\mathbd{f}\|_{\Xsp}=\max_{i=1,\,\ldots,\,n}\{\|f_i\|_{\Hinf(\domD)}\}$.
The goal is to show $\mathbd{y}\in\Ysp$,
but it is not easy because $\Ysp$ is not a Banach space.
For this reason, firstly $\mathbd{y}\in\Xsp$ is shown
($\Xsp$ is a Banach space),
and by using the result, $\mathbd{y}\in\Ysp$ is shown.
Let us introduce the integral operator
$\Jint:\Xsp \to \Xsp$ as
$\Jint [\mathbd{f}](t) = \int_a^t \mathbd{f}(s)\diff s$,
and $\Vol:\Xsp \to \Xsp$ as
\begin{equation*}
\Vol[\mathbd{f}](t) = \int_a^t K(s)\mathbd{f}(s)\diff s,
\end{equation*}
where $K$ satisfies the assumption (SE1) or (DE1).
If the operator is multiplied repeatedly,
it becomes a contraction map.

\begin{lem}
Let the assumption {\rm (SE1)} be fulfilled.
Then it holds for all positive integers $m$ and $z\in\SEt(\domD_d)$
that
\[
 |\Vol^m[\mathbd{f}](z)|
\leq
 \frac{\{n L(b-a)^{2\alpha-1}c_1\Bfunc(\psi_1(x),\alpha,\alpha)\}^m}
      {m!}\|\mathbd{f}\|_{\Xsp} [1,\,1,\,\ldots,\,1]^{\trans},
\]
where $x=\Re[\SEt(z)]$, $\psi_1(x)=(\tanh(x/2) + 1)/2$,
$\Bfunc(x,\alpha,\beta)$ is the incomplete beta function,
$L$ is the constant in Eq.~\eqref{Leq:LC-bounded-by-Q},
and $c_1$ is a constant depending only on $d$.
\end{lem}
\begin{lem}
\label{lem:Vol-contraction-DE}
Let the assumption {\rm (DE1)} be fulfilled.
Then it holds for all positive integers $m$ and $z\in\DEt(\domD_d)$
that
\[
 |\Vol^m[\mathbd{f}](z)|
\leq
 \frac{\{n L(b-a)^{2\alpha-1}c_2\Bfunc(\psi_2(x),\alpha,\alpha)\}^m}
      {m!}\|\mathbd{f}\|_{\Xsp} [1,\,1,\,\ldots,\,1]^{\trans},
\]
where $x=\Re[\DEt(z)]$, $\psi_2(x)=(\tanh(\pi \sinh(x)/2) + 1)/2$,
$L$ is the constant in Eq.~\eqref{Leq:LC-bounded-by-Q},
and $c_2$ is a constant depending only on $d$.
\end{lem}
These lemmas are straightforward extension
from the existing ones~\cite[Lemmas~5.4 and~5.6]{okayama11:_theor},
and the proofs are omitted.
Then in both cases it holds that
\[
 \|\Vol^m \mathbd{f}\|_{\Xsp}\leq
\frac{\{n L (b-a)^{2\alpha - 1}c_i \Bfunc(1,\alpha,\alpha)\}^m}
      {m!}\|\mathbd{f}\|_{\Xsp},
\]
and thus for sufficiently large $m$, $\Vol^m$ is a contraction map,
from which we have the next theorem.

\begin{thm}
\label{thm:Hinf-Vol-SE}
Let the assumptions {\rm (SE1)} and {\rm (SE2)} be fulfilled.
Then Eq.~\eqref{Vol} has a unique solution $\mathbd{y}\in\Xsp$,
i.e.,
$y_i\in\Hinf(\SEt(\domD_d))$ for $i=1,\,\ldots,\,n$.
\end{thm}
\begin{thm}
\label{thm:Hinf-Vol-DE}
Let the assumptions {\rm (DE1)} and {\rm (DE2)} be fulfilled.
Then Eq.~\eqref{Vol} has a unique solution $\mathbd{y}\in\Xsp$,
i.e.,
$y_i\in\Hinf(\DEt(\domD_d))$ for $i=1,\,\ldots,\,n$.
\end{thm}
\begin{proof}
Before applying the contraction mapping theorem,
the only thing we have to show is
(SE2)/(DE2) $\Rightarrow$ $\Jint \mathbd{g}\in\Xsp$,
which is done by Lemmas~\ref{lem:SE-Indef-in-MC}
and~\ref{lem:DE-Indef-in-MC}
(since $\MC_{\alpha}(\domD)\subset \Hinf(\domD)$).
\end{proof}

The next lemma is a result for SE,
which completes the proof of Theorem~\ref{thm:Hinf-Vol-SE}.

\begin{lem}[Stenger~{\cite[Theorem 4.1.3]{stenger93:_numer}}]
\label{lem:SE-Indef-in-MC}
Let $gQ \in\LC_{\alpha}(\SEt(\domD_d))$,
and set $q(t)=\int_a^t g(s)\diff s$.
Then $q\in\MC_{\alpha}(\SEt(\domD_d))$.
\end{lem}

In the case of DE
(for Theorem~\ref{thm:Hinf-Vol-DE}), we need the next lemma.

\begin{lem}[Okayama et al.~{\cite[Lemma A.4]{okayama10:_sinc}}]
\label{Lem:Bound-complex-DEtrans}
For $x\in\mathbb{R}$ and $y\in (-\pi/2,\,\pi/2)$, it holds that
\[
\psi_2(x):=
 \frac{1}{2}\tanh\left(\frac{\pi\cos y}{2}\sinh x\right) +\frac{1}{2}
\leq \left|
 \frac{1}{2}\tanh\left(\frac{\pi}{2}\sinh (x+\imnum y)\right) + \frac{1}{2}
\right|.
\]
\end{lem}

Using Lemma~\ref{Lem:Bound-complex-DEtrans},
we show the following lemma
(DE version of Lemma~\ref{lem:SE-Indef-in-MC}).
\begin{lem}
\label{lem:DE-Indef-in-MC}
Let $gQ \in\LC_{\alpha}(\DEt(\domD_d))$,
and set $q(t)=\int_a^t g(s)\diff s$.
Then $q\in\MC_{\alpha}(\DEt(\domD_d))$.
\end{lem}
\begin{proof}
Clearly $q\in\Hinf(\DEt(\domD_d))$ holds.
Let us show the H\"{o}lder continuity at the endpoint $a$
(showing it at $b$ is omitted, because it is quite similar).
Putting $n=1$,\,$f\equiv 1$,\,$K=g$ in Lemma~\ref{lem:Vol-contraction-DE},
we have
\[
\left|\int_a^{z} g(w)\diff w\right|
\leq L(b-a)^{2\alpha - 1} c_2 \Bfunc(\psi_2(x),\alpha,\alpha).
\]
Then it holds that
\[
 |q(z)-q(a)|
=\left|\int_a^z g(w)\diff w\right|
\leq L(b-a)^{2\alpha - 1} c_d \Bfunc(\psi_2(x),\alpha,\alpha)
\leq L(b-a)^{2\alpha - 1} c_d \Bfunc(1,\alpha,\alpha) \{\psi_2(x)\}^{\alpha}.
\]
Furthermore, from Lemma~\ref{Lem:Bound-complex-DEtrans},
it holds that
\[
 (b-a)\psi_2(x) \leq \left|
 \frac{b-a}{2}\tanh\left(\frac{\pi}{2}\sinh (x+\imnum y)\right) + \frac{b-a}{2}
\right|
=\left|\DEt(x+\imnum y) - a\right| = |z-a|.
\]
Thus there exists a constant $\tilde{L}$
such that
$|q(z)-q(a)|\leq \tilde{L}|z-a|^{\alpha}$
for all $z\in\DEt(\domD_d)$.
\qed
\end{proof}

Now showing $\mathbd{y}\in\Xsp$ is finished.
For $\mathbd{y}\in\Ysp$,
what is left is to show the H\"{o}lder continuity.
In view of the right hand side of Eq.~\eqref{Vol},
clearly $\mathbd{y}$ (on the left hand side)
is H\"{o}lder continuous of $\alpha$-order.
Hence Theorems~\ref{thm:SE-sol-regularity}
and~\ref{thm:DE-sol-regularity} are established.

\subsection{Proofs on convergence of the numerical solutions}

\subsubsection{SE/DE-Sinc-Nystr\"{o}m method}

Firstly, the SE-Sinc-Nystr\"{o}m method is considered.
Notice that in this subsection,
set $\Csp=\{C([a,\,b])\}^n$, and
all operators here are discussed on this function space.
Let us introduce the operator
$\JnSE$, which is an approximation of $\Jint$, as
\begin{align*}
\JnSE[\mathbd{f}](t) &= \sum_{j=-N}^N \mathbd{f}(\tSE_j)\SEtDiv(jh)J(j,h)(\SEtInv(t)),
\end{align*}
and $\VolSEn$ as $\VolSEn[\mathbd{f}](t)=\JnSE[K\mathbd{f}](t)$.
Then consider the following three equations:
\begin{align*}
(\mathcal{I} - \Vol)\mathbd{y} &= \mathbd{r} + \Jint \mathbd{g}
&& (\text{Eq.~\eqref{Vol}}),\\
(\mathcal{I} - \VolSEn)\yn &= \mathbd{r} + \JnSE \mathbd{g}
&& (\text{Eq.~\eqref{Def-UsolSE}}),\\
 (I_n\otimes I_N - I_n\otimes \{h I^{(-1)}_N D_N^{\textSE}\}[K_{ij}^{\textSE}])
\mathbd{Y}^{\textSE}
& = I_n\otimes \{h I^{(-1)}_N D_N^{\textSE}\}\mathbd{G}^{\textSE}+\mathbd{R}
&& (\text{Eq.~\eqref{linear-eq-SE}}).
\end{align*}
Using the standard arguments (e.g., see~\cite[Lemma~6.1]{okayama11:_theor}),
we can see the unique solvability
of Eq.~\eqref{linear-eq-SE} is equivalent to
that of Eq.~\eqref{Def-UsolSE}.
If the unique solvability of Eq.~\eqref{Def-UsolSE} is shown,
i.e., $(\mathcal{I} - \VolSEn)^{-1}$ exists,
we have
\begin{align*}
\mathbd{y} - \yn
&=(\mathcal{I}-\VolSEn)^{-1}\left\{
(\mathcal{I}-\VolSEn)\mathbd{y} - (\mathbd{r}+\JnSE \mathbd{g})
\right\}\\
&=(\mathcal{I}-\VolSEn)^{-1}\left\{
(\mathbd{r}+\Jint\mathbd{g} +\Vol\mathbd{y}) - \VolSEn\mathbd{y}
-\mathbd{r}-\JnSE\mathbd{g}
\right\}\\
&=(\mathcal{I}-\VolSEn)^{-1}\left\{
(\Jint\mathbd{g} - \JnSE\mathbd{g}) + (\Vol\mathbd{y} - \VolSEn\mathbd{y})
\right\},
\end{align*}
and finally using Theorem~\ref{Thm:SE-Sinc-indef},
the desired error estimate (Theorem~\ref{thm:converge-SE-Sinc-Nystroem})
is obtained.
Therefore, what is left is to show
the existence and boundedness of $(\mathcal{I} - \VolSEn)^{-1}$.
For the purpose, the next theorem is useful.

\begin{thm}[Atkinson~{\cite[Theorem~4.1.1]{atkinson97:_numer_solut}}]
\label{Thm:Atkinson-Nystroem}
Assume the following four conditions:
\begin{enumerate}
 \item Operators $\mathcal{X}$ and $\mathcal{X}_n$
are bounded operators on $\Csp$ to $\Csp$.
 \item The operator $(\mathcal{I} - \mathcal{X}):\Csp\to\Csp$
has a bounded inverse
$(\mathcal{I} - \mathcal{X})^{-1}:\Csp\to\Csp$.
 \item The operator $\mathcal{X}_n$ is compact on $\Csp$.
 \item The following inequality holds:
\begin{equation*}
\|(\mathcal{X}-\mathcal{X}_n)\mathcal{X}_n\|_{\mathcal{L}(\Csp,\Csp)}
<\frac{1}{\|(\mathcal{I} - \mathcal{X})^{-1}\|_{\mathcal{L}(\Csp,\Csp)}}.
\end{equation*}
\end{enumerate}
Then $(\mathcal{I} - \mathcal{X}_n)^{-1}$ exists
as a bounded operator on $\Csp$ to $\Csp$, with
\begin{equation}
\|(\mathcal{I} - \mathcal{X}_n)^{-1}\|_{\mathcal{L}(\Csp,\Csp)}
\leq \frac{1 + \|(\mathcal{I} - \mathcal{X})^{-1}\|_{\mathcal{L}(\Csp,\Csp)}
\|\mathcal{X}_n\|_{\mathcal{L}(\Csp,\Csp)}}
{1 - \|(\mathcal{I} - \mathcal{X})^{-1}\|_{\mathcal{L}(\Csp,\Csp)}\|(\mathcal{X}-\mathcal{X}_n)\mathcal{X}_n\|_{\mathcal{L}(\Csp,\Csp)}}.
\label{InEq:Bound-Inverse-Op}
\end{equation}
\end{thm}

We need to show the four conditions in this theorem
as $\mathcal{X}=\Vol$ and $\mathcal{X}_n=\VolSEn$.
The first condition clearly holds,
and the second condition is known as a classical result.
The third condition immediately follows from the Arzel\'{a}--Ascoli theorem.
The fourth condition is shown by the next lemma,
which is
straightforward extension
from the existing one~\cite[Lemma~6.5]{okayama11:_theor}.

\begin{lem}
Let the assumption {\rm (SE1)} be fulfilled.
Then there exists a constant $C$ independent of $N$ such that
\begin{equation*}
\|(\Vol - \VolSEn)\VolSEn\|_{\mathcal{L}(\Csp,\Csp)}
\leq \frac{C}{\sqrt{N}}.
\end{equation*}
\end{lem}

Furthermore,
$\|\VolSEn\|_{\mathcal{L}(\Csp,\Csp)}$ is uniformly bounded,
since $\VolSEn \mathbd{f}$ converges to $\Vol\mathbd{f}$
for any $\mathbd{f}\in\Csp$.
Thus, from Eq.~\eqref{InEq:Bound-Inverse-Op}, we obtain
the desired result: $(\mathcal{I} - \VolSEn)^{-1}$
exists and uniformly bounded
for all sufficiently large $N$.
This completes the proof of Theorem~\ref{thm:converge-SE-Sinc-Nystroem}
(the SE-Sinc-Nystr\"{o}m method).

The proof for the DE-Sinc-Nystr\"{o}m method goes on
in exactly the same way.
Let us introduce the operator
$\JnDE$ as
\begin{align*}
\JnDE[\mathbd{f}](t) &= \sum_{j=-N}^N \mathbd{f}(\tDE_j)\DEtDiv(jh)J(j,h)(\DEtInv(t)),
\end{align*}
and $\VolDEn$ as $\VolDEn[\mathbd{f}](t)=\JnDE[K\mathbd{f}](t)$.
The difference from the SE is the next lemma,
which is also straightforward extension
from the existing one~\cite[Lemma~6.9]{okayama11:_theor}.
\begin{lem}
Let the assumption {\rm (DE1)} be fulfilled.
Then there exists a constant $C$ independent of $N$ such that
\begin{equation*}
\|(\Vol - \VolDEn)\VolDEn\|_{\mathcal{L}(\Csp,\Csp)}
\leq C\left(\frac{\log(2 d N/\alpha)}{N}\right)^2.
\end{equation*}
\end{lem}
This completes the proof of Theorem~\ref{thm:converge-DE-Sinc-Nystroem}
(the DE-Sinc-Nystr\"{o}m method).

\subsubsection{SE/DE-Sinc-collocation method}

Let us consider the SE-Sinc-collocation method first.
Notice the relation
$\tilde{y}_i^{(N)}(t) = \ProjSE[y_i^{(N)}](t)$,
where $y_i^{(N)}$ is
the solution of the SE-Sinc-Nystr\"{o}m method
(see Eq.~\eqref{Def-UsolSE}),
and $\tilde{y}_i^{(N)}$ is the solution of the SE-Sinc-collocation method
(see Eq.~\eqref{Def-SE-Sinc-colloc}).
Then we have
\begin{equation}
\|y_i - \tilde{y}_i^{(N)}\|_{\Csp}
\leq \|y_i - \ProjSE {y}_i\|_{\Csp}
+\|\ProjSE\|_{\mathcal{L}(\Csp,\Csp)}\|y_i - {y}_i^{(N)}\|_{\Csp}.
\label{leq:SE-colloc-estim}
\end{equation}
The first term on the right hand side
can be estimated by Theorem~\ref{Thm:SE-Sinc-Base}.
On the second term,
use Theorem~\ref{thm:converge-SE-Sinc-Nystroem}
for $\|y_i - {y}_i^{(N)}\|_{\Csp}$,
and use the next lemma
to obtain $\|\ProjSE\|_{\mathcal{L}(\Csp,\Csp)}\leq C\log (N+1)$.
\begin{lem}[Stenger~{\cite[p.\,142]{stenger93:_numer}}]
\label{Lem:Sinc-Real-Sum}
Let $h>0$. Then it holds that
\begin{equation*}
\sup_{\xi\in\mathbb{R}}\sum_{j=-N}^N |S(j,h)(\xi)|
\leq \frac{2}{\pi}(3 + \log N).
\end{equation*}
\end{lem}
This completes the proof of Theorem~\ref{thm:SE-Sinc-converge}
(the SE-Sinc-collocation method).

The proof for the DE-Sinc-collocation method goes on
in exactly the same way.
By using the relation $\tilde{y}_i^{(N)}(t) = \ProjDE[y_i^{(N)}](t)$,
we have the similar inequality as Eq.~\eqref{leq:SE-colloc-estim}
(just replace SE with DE).
By estimating each term
via Theorems~\ref{Thm:DE-Sinc-Base} and~\ref{thm:converge-DE-Sinc-Nystroem}
and Lemma~\ref{Lem:Sinc-Real-Sum},
Theorem~\ref{thm:DE-Sinc-converge} is established.






\begin{small}
\bibliographystyle{model1b-num-names}
\bibliography{Sinc-colloc-InitVal}
\end{small}






\end{document}